\documentclass[11pt]{article}

\usepackage[utf8]{inputenc}
\usepackage[T1]{fontenc}
\usepackage{lmodern}

\usepackage{cite}
\usepackage{lineno}

\usepackage{ifthen}
\usepackage{makeidx}
\usepackage{amsmath,amssymb,amstext} 
\usepackage{graphicx} 
\usepackage{fullpage}
\usepackage{datetime}
\numberwithin{equation}{section}
\usepackage{enumitem}
\usepackage{multirow}
\numberwithin{figure}{section}

\usepackage{booktabs}

\usepackage{algorithm,algorithmic} 

\numberwithin{table}{section}
\numberwithin{algorithm}{section}

\usepackage[letterpaper=true,pagebackref=true]{hyperref} 

\hypersetup{
    colorlinks=true,        
    linkcolor=blue,         
    citecolor=green,        
    filecolor=magenta,      
    urlcolor=cyan           
}

\usepackage{comment}

\usepackage{float}
\usepackage{bbm}
\usepackage{exscale}
\usepackage{tabularx}
\usepackage{syntonly}
\usepackage{lineno}

\usepackage{amssymb}
\usepackage{amsmath}
\usepackage{amsthm}
\usepackage{latexsym}
\usepackage{makeidx}
\usepackage{longtable}
\usepackage[noabbrev]{cleveref}
\numberwithin{equation}{section}

\makeindex
\usepackage{chngcntr}
\usepackage{bm}

\def\Rout{R^{\text{out}}}
\def\Yout{Y^{\text{out}}}
\def\Zout{Z^{\text{out}}}
\def\xapprox{x^{\text{approx}}}

\def\R{\mathbb{R}}
\def\Sc{\mathbb{S}}
\def\Sn{\Sc^n}

\def\Snp{\Sc_+^n}

\def\Snpp{\Sc_{++}^n}

\def\Rn{\mathbb{R}^n}

\def\eqref#1{{\normalfont(\ref{#1})}}

\newenvironment{noteH}{\begin{quote}\color{blue} \small\sf HenryW $\clubsuit$~}{\end{quote}}

\def\eqref#1{{\normalfont(\ref{#1})}}

\newtheorem{theorem}{Theorem}[section]

\newtheorem{thm}[theorem]{Theorem}

\newtheorem{corollary}[theorem]{Corollary}

\newtheorem{remark}[theorem]{Remark}

\newtheorem{lemma}[theorem]{Lemma}

\crefname{thm}{Theorem}{Theorems}
\Crefname{thm}{Theorem}{Theorems}
\crefname{problem}{Problem}{Theorems}
\Crefname{problem}{Problem}{Theorems}
\crefname{conjecture}{Conjecture}{Theorems}
\Crefname{conjecture}{Conjecture}{Theorems}
\crefname{proposition}{Proposition}{Propositions}
\Crefname{proposition}{Proposition}{Propositions}
\crefname{prop}{Proposition}{Propositions}
\Crefname{prop}{Proposition}{Propositions}
\crefname{cor}{Corollary}{Corollaries}
\Crefname{cor}{Corollary}{Corollaries}
\crefname{lem}{Lemma}{Lemmas}
\Crefname{lem}{Lemma}{Lemmas}
\theoremstyle{definition}
\crefname{definition}{definition}{definitions}
\Crefname{definition}{Definition}{Definitions}
\crefname{defn}{definition}{definitions}
\Crefname{defn}{Definition}{Definitions}
\crefname{remark}{Remark}{Remarks}
\Crefname{remark}{Remark}{Remarks}
\crefname{rmk}{Remark}{Remarks}
\Crefname{rmk}{Remark}{Remarks}
\crefname{example}{Example}{Examples}
\Crefname{example}{Example}{Examples}
\crefname{align}{}{}
\Crefname{align}{}{}
\crefname{equation}{}{}
\Crefname{equation}{}{}

\newcommand{\textdef}[1]{\textit{#1}\index{#1}}

\newcommand{\<}{\langle}
\renewcommand{\>}{\rangle}

\newcommand{\bS}{\mathbb{S}}
\newcommand{\cN}{{\mathcal N} }
\newcommand{\cP}{{\mathcal P}}
\newcommand{\cZ}{{\mathcal Z} }
\newcommand{\cR}{{\mathcal R} }
\newcommand{\cY}{{\mathcal Y} }
\newcommand{\cJ}{{\mathcal J} }
\newcommand{\cL}{{\mathcal L} }
\newcommand{\cV}{{\mathcal V} }
\newcommand{\cC}{{\mathcal C} }

\newcommand{\cU}{{\mathcal U} }

\newcommand{\ZZ}{{\mathcal Z} }

\newcommand{\LL}{{\mathcal L} }
\newcommand{\RR}{{\mathcal R} }
\newcommand{\NN}{{\mathcal N} }

\newcommand{\IQP}{\textbf{IQP }}
\newcommand{\IQPp}{\textbf{IQP}}

\newcommand{\PRSM}{\textbf{PRSM }}
\newcommand{\PRSMp}{\textbf{PRSM}}
\newcommand{\rPRSM}{\textbf{rPRSM }}
\newcommand{\rPRSMp}{\textbf{rPRSM}}

\newcommand{\SCPp}{\textbf{SCP}}
\newcommand{\SCP}{\textbf{SCP}\,}
\newcommand{\col}{\textbf{col}\,}
\newcommand{\DNNp}{\textbf{DNN}}
\newcommand{\DNN}{\textbf{DNN}\,}

\newcommand{\SDP}{\textbf{SDP}\,}

\newcommand{\SDPp}{\textbf{SDP}}
\newcommand{\FR}{\textbf{FR}\,}
\newcommand{\FRp}{\textbf{FR}}

\newcommand{\A}{{\mathcal A}}

\newcommand{\bbm}{\begin{bmatrix}}
\newcommand{\ebm}{\end{bmatrix}}
\newcommand{\bem}{\begin{pmatrix}}
\newcommand{\eem}{\end{pmatrix}}
\newcommand{\beq}{\begin{equation}}
\newcommand{\beqs}{\begin{equation*}}
\newcommand{\bet}{\begin{table}}
\newcommand{\eeq}{\end{equation}}
\newcommand{\eeqs}{\end{equation*}}
\newcommand{\beqr}{\begin{eqnarray}}

\DeclareMathOperator{\Null}{null}

\DeclareMathOperator{\range}{range}

\DeclareMathOperator{\trace}{{trace}}

\DeclareMathOperator{\BlkDiag}{{BlkDiag}}
\DeclareMathOperator{\tr}{{trace}}
\DeclareMathOperator{\diag}{{diag}}
\DeclareMathOperator{\Diag}{{Diag}}

\DeclareMathOperator{\rank}{{rank}}

\newcommand{\nc}{\newcommand}
\nc{\arrow}{{\rm arrow\,}}
\nc{\Arrow}{{\rm Arrow\,}}
\nc{\BoDiag}{{\rm B^0Diag\,}}
\nc{\bodiag}{{\rm b^0diag\,}}

\nc{\Mm}{{\mathcal M}^{m} }
\nc{\Mmn}{{\mathcal M}^{mn} }
\nc{\Mnr}{{\mathcal M}_{nr} }
\nc{\Mnmr}{{\mathcal M}_{(n-1)r} }
\nc{\kwqqp}{Q{$^2$}P\,}
\nc{\kwqqps}{Q{$^2$}Ps}
\def\argmin{\mathop{\rm argmin}}

\nc{\notinaho}{(X,S)\in \overline{AHO}(\A)}
\nc{\inaho}{(X,S)\in AHO(\A)}

\newcommand{\bea}{\begin{eqnarray}}%
\newcommand{\eea}{\end{eqnarray}}%
\newcommand{\beas}{\begin{eqnarray*}}%
\newcommand{\eeas}{\end{eqnarray*}}%
\newcommand{\Rnn}{\R^{n \times n}}%

{}

\newcommand{\Hnp}[1][]{\,\mathbb{H}_+^{\ifthenelse{\equal{#1}{}}{n}{#1}}}
\newcommand{\Hn}[1][]{\,\mathbb{H}^{\ifthenelse{\equal{#1}{}}{n}{#1}}}
\newcommand{\Dn}[1][]{\,\mathbb{D}^{\ifthenelse{\equal{#1}{}}{n}{#1}}}

\newcommand*\samethanks[1][\value{footnote}]{\footnotemark[#1]}  

\begin{document}

\bibliographystyle{plain}
\title{
A Peaceman-Rachford Splitting Method for \\
the Protein Side-Chain Positioning Problem
}
  \author{
Forbes Burkowski\thanks{Professor Emeritus, Cheriton School of Computer Science,  University of Waterloo, Canada}
\and
Haesol Im
\thanks{Department of Combinatorics and Optimization, Faculty of Mathematics, University of Waterloo, Canada, Research supported by NSERC}
\and
\href{http://www.math.uwaterloo.ca/~hwolkowi/}
{Henry Wolkowicz}
\samethanks
}
\date{ 
		\today\\
}
\maketitle




\vspace{-.4in}

\begin{abstract}
We formulate a doubly nonnegative (\DNNp) relaxation of the protein
side-chain positioning (\SCPp) problem.
We inherit the natural splitting of variables that
stems from the facial reduction technique in the semidefinite
relaxation. 
We solve the relaxation using a variant of the Peaceman-Rachford splitting method.
Our numerical experiments show that we solve \emph{almost all} instances 
of the NP-hard \SCP problem to \emph{optimality}.
\end{abstract}


\tableofcontents
\listoffigures
\listoftables

\vspace{-.2in}

\section{Introduction}
\label{sec:intro}

\index{\SCPp, side-chain positioning}
\index{\IQPp, integer quadratic problem}
\index{\DNNp, doubly nonnegative relaxation}

The \emph{protein} \textdef{side-chain positioning (\SCPp)} problem is one of the most important subproblems of the protein structure prediction problem.
We formulate the \SCP problem as an integer program and derive its 
doubly nonnegative, \DNNp, relaxation. 
We then use a variation of the \textdef{Peaceman-Rachford splitting method (\PRSMp)} to solve the \DNN relaxation.

The applications of \SCP extend to ligand binding \cite{LoogerLorenL2003Cdor, LAUDET200237} and protein-protein docking with backbone flexibility \cite{WangChu2007PDwB, MarzeNicholasA2018Efbp}.
A protein is a macromolecule consisting of a long main chain backbone that provides a set of anchors for a sequence of amino acid side-chains. The backbone is comprised of a repeating triplet of atoms (nitrogen, carbon, carbon) with the central carbon atom being designated as the alpha carbon.  An amino acid side-chain is a smaller (1 to 18 atoms) side branch that is anchored to an alpha carbon.  The positions of the atoms in a side-chain can be established by knowing the 3D position of its alpha carbon and the dihedral angles defined by atoms in the side-chain.  The number of dihedral angles varies from 1 to 4 depending on the length of the side-chain. This is true for 18 of the 20 amino acids with glycine and alanine being exceptions because their low atom counts preclude dihedral angles.

It has been observed that the values of dihedral angles are not uniformly distributed. They tend to form clusters with cluster centers that are equally separated (+60, 180, -60). Consequently, if the dihedral angles are unknown we at least have a reasonable estimate of their values by appealing to these discretized values.  With this strategy being applied, a side-chain with one dihedral angle would have three possible sets of positions for its atoms. We refer to each set of atomic positions as a rotamer.  A side-chain with two dihedral angles will have 3 times 3 or 9 different arrangements of the atoms (i.e. 9 rotamers).  Three dihedral angles will result in 27 rotamers and four dihedral angles will give 81 rotamers.

In the \SCP problem we are given a fixed backbone and a designation of the amino acid type for each alpha carbon.  To solve the problem it is required that each amino acid is assigned a particular rotameric setting with the objective of avoiding any collisions with neighbouring amino acids that are given their rotameric settings. Avoiding collisions will lower the overall energy of the protein and, in fact, even with all possible collisions circumvented we want to have an energy evaluation that is minimal.

The \SCP problem is proven to be NP-hard \cite{Akutsu1997180}.
The nature of the \SCP problem has motivated the development of many heuristic based algorithms \cite{CanutescuAdrianA2003Agaf, DesmetJohan1992Tdet, Bahadur2004ProteinSP, SamudralaR1998Dosc, BowerMichaelJ1997Pops, XuJinbo2006Faaa} and many of these approaches rely on the graph structure of the problem. 
Other approaches for solving \SCP problems have been proposed.
These range from probabilistic approaches \cite{LeeChristopher1994PPME, HolmLiisa1991Dafg, ShenkinPeterS1996Paeo}, 
integer programming \cite{MR1942452, KingsfordCarletonL2005Saas,
AlthausErnst2002Acat}, to 
semidefinite programming \cite{MR2098320,ForbesVrisWolk:11}. 
Our approach is based on a semidefinite programming relaxation.
Given a rotamer library, the \SCP problem can be formulated as an 
\textdef{integer quadratic problem (\IQPp)}.
We then obtain a \textdef{semidefinite programming (\SDPp)} relaxation
to the \IQP via a lifting of variables and \textdef{facial reduction (\FRp)}. 
We finally obtain a \textdef{doubly nonnegative (\DNNp)} relaxation
by adding nonnegativity constraints and  some additional constraints to the \SDP that help strengthen our relaxation.

The facial reduction originating from the \SDP relaxation delivers a natural splitting of variables. 
This elegant splitting of variables fits into the framework of the splitting methods. 
The framework gives an efficient procedure of engaging constraints that
are difficult to process simultaneously, see
e.g.,~\cite{GHILW:20,LiPongWolk:19,OliveiraWolkXu:15}.
We solve the \DNN relaxation using a variation of the so-called \textdef{Peaceman-Rachford splitting method (\PRSMp)}.
Using the \PRSMp, we examine the strength of our approach in the numerical experiments.
The usage of splitting method for the \DNN relaxation allows for 
an effective treatment for handling implicit redundant constraints and 
 the ill-posed data that stems from collisions between rotamers.

\index{\DNNp, doubly nonnegative relaxation}
\index{\FRp, facial reduction}
\index{\SDPp, semidefinite programming}

\subsection{Notation}

\index{$\Rn$}
\index{$\R^{m \times n}$}
\index{$\Snp$}
\index{$\Snpp$}
\index{$X \succeq 0$}
\index{$X \succ 0$}
\index{$\<\cdot,\cdot\>$}
\index{$\cN_\cC$}
\index{$\BlkDiag$}
\index{$\diag$}
\index{$\Diag$}

We let $\Rn, \R^{m \times n}$ denote the standard real Euclidean spaces;
$\Sn$ denotes the Euclidean space of $n$-by-$n$ real symmetric matrices; 
$\Snp$ ($\Snpp$, resp) denotes the cone of $n$-by-$n$ positive semidefinite (definite, resp) matrices. 
We write $X\succeq 0$ if $X\in \Snp$, and $X\succ 0$ if $X\in \Snpp$.
We use $\range(X)$ and $\Null(X)$ to denote the range of~$X$ and the null space of $X$, respectively.
Given $X \in \Rnn$, we use $\trace (X)$ to denote the \textdef{trace} of $X$.
Given two matrices $X,Y \in \R^{m \times n}$, we let $\<X,Y\> = \trace (XY^T)$ denote the usual trace inner product between $X$ and $Y$; $X \circ Y$ denote the element-wise, or Hadamard, product of $X$ and $Y$.
Given a closed convex set $\cC$ in a Euclidean space, we let $\cN_{\cC}(x)$ denote the \textdef{normal cone} at $x \in \cC$ with respect to $\cC$. 
Given $X \in \Rnn$, we use $\diag(X)$ to denote the vector formed from the diagonal entries of~$X$. 
Then $\Diag(v) = \diag^*(v)$ is the adjoint linear transformation that
forms the diagonal matrix from the vector $v$.
Given a collection of matrices $\{A_i\}_{i=1}^{m}$, we let $\BlkDiag(A_1,\ldots,A_m)$ denote the block diagonal matrix with the $i$-th diagonal block $A_i$.
We let \textdef{$\bar{e}_{n}$} denote the $n$-dimensional vector with each entry set to $1$ and we omit the subscript when the dimension is clear. 
Given a positive integer $m$, 
we often use the notation $[m]$ to mean the set of positive integers $\{1,\ldots, m\}$.

\index{$[m] = \{1,\ldots, m\}$}


\subsection{Contributions and Outline}

We present the process for formulating the model in~\Cref{sec:Model_Derivation}. 
We formulate the \SCP problem as an integer quadratic program, \IQPp,
and obtain the \SDP and \DNN relaxations.
The derivation for the \SDP relaxation is first presented in~\cite{ForbesVrisWolk:11} via Lagrangian relaxation.
Here, we present a much simpler derivation for the \SDP relaxation via direct lifting of the variables. 
In \Cref{sec:Algorithm} we present a variation of the \PRSM for solving the \DNN relaxation as well as our strategies to obtain upper and lower bounds to the \SCP problem.
We show that the splitting method engages implicit redundant constraints safely that arise from the facial reduction.
In~\Cref{sec:Numerical_Experiments} we use the real-world data from the Protein Data Bank\footnote{\url{https://www.rcsb.org/}} to 
examine the strength of our approach.
We show that the usage of splitting method to the \DNN relaxation effectively handles collisions between rotamers that are indicated by large values in the data. 
Moreover, the numerical experiments demonstrate that our approach \emph{provably}
solves \emph{almost all} instances\footnote{Out of $131$ test problems,
one problem had a positive
gap; five other problems had gaps of approximately~$10^{-6}$.}
to the global optimum of the NP-hard protein \SCP problem.


\section{Model Derivation}
\label{sec:Model_Derivation}

The goal of this section is to obtain the \DNN relaxation of the \SCP problem. 
We start by presenting a formulation of the \SCP
as an \IQP in~\Cref{sec:Problem_Formulation}.  We then
derive its \SDP relaxation in~\Cref{sec:SDP_Relaxation}.
We continue the derivation by identifying redundant constraints in the \IQP and in the 
\SDP relaxation in order to obtain a complete (stable) \DNN relaxation in \Cref{sec:DNN_Relaxation}.

\subsection{Problem Formulation as \IQPp}
\label{sec:Problem_Formulation}

\index{\SCPp, side-chain positioning}
\index{$m_i$}
\index{$n_0$, total number of rotamers}
\index{$\cV_i$, $i$-th rotamer set}
\index{$i$-th rotamer set, $\cV_i$}
\index{rotamer}

We are given a collection of disjoint sets $\cV _i, i = 1,\ldots,p$.
Each set $\cV _i$ has $m_i$ members and we index its members
\[
\cV_i := \{ v_i^1, v_i^2, \ldots, v_i^{m_i} \} , \ \text{ for all }    i =1 ,\ldots, p .
\]
We call each set $\cV_i$ a \emph{rotamer set} and its members \emph{rotamers}.
We use $n_0 = \sum_{i=1}^p m_i$ and~\textdef{$\cV = \cup_{i=1}^p \cV_i$}.
The \textdef{protein side-chain positioning problem} seeks to
\begin{enumerate}
\item select \emph{exactly one}~rotamer $v_{i}^j$, from each set $\cV_{i}$, where $j \in [m_i]$ (see~\Cref{fig:problem_illustration}\footnote{
$\cV_i$ indicates the $i$-th rotamer set and $v_i^j$ indicates the $j$-th candidate in the $i$-th rotamer set $\cV_i$.}); and
\item minimize the
sum of the weights (energy) determined by chosen rotamers, and
the energy between each chosen rotamer and the 
backbone.
\end{enumerate}
\begin{figure}[h!]
\centering
\hspace{0.8cm}
\includegraphics[trim= 0 110 0 100,clip,height=5cm]{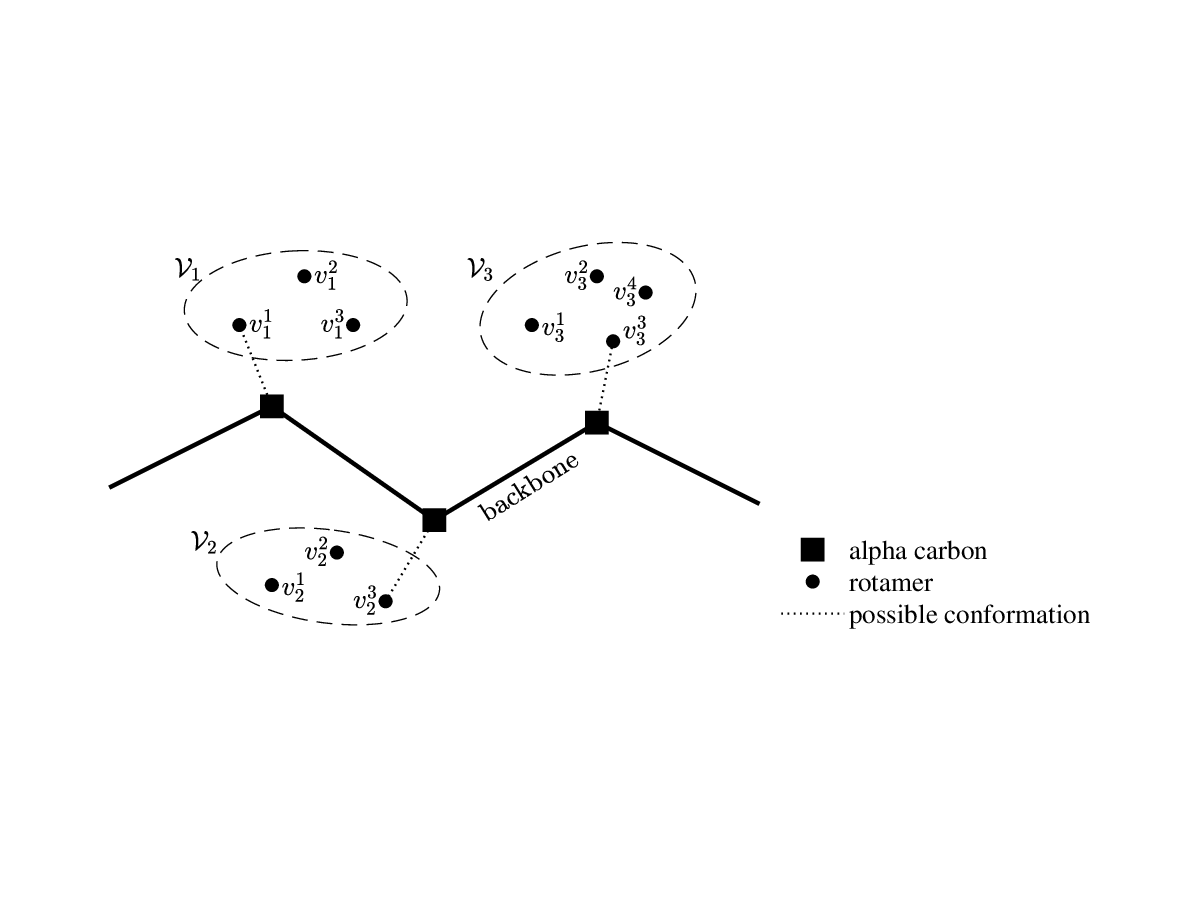}
\vspace{-0.3cm}
\caption{A Diagram of the Protein Side-Chain Positioning Problem}
\label{fig:problem_illustration}
\end{figure}

Viewing the rotamers as a set of nodes of a graph, we can realize the \SCP problem as a discrete optimization problem over a graph.
We construct a matrix $E\in \bS^{n_0}$ to record the energy values between rotamers and the backbone. 
We use the matrix entries $E_{uv}$, with $u\neq v$, 
to denote the edge weights between two distinct rotamers (nodes); 
while the diagonal entries $E_{uu}$ denote the weight between the rotamer $u$ and the backbone. 
This yields a symmetric matrix $E$, where $E_{uv}=\infty$ if both
rotamers $u,v$ are in the same set. We note that the multiplication
 $0 \cdot \infty = 0$ when adding up the weights (energies).
Alternatively, we can set these weights to $0$ and add a constraint to
choose exactly one rotamer from each set, which is what we do.
Thus each diagonal block of $E$, of size $m_{i}$, can be assumed to be
a diagonal matrix. We can make this simplification without loss of generality since we are looking to only choose one rotamer per set $\cV_{i}$.
\index{$E$, edge weights}
\index{edge weights, $E$}
\index{$\cV$}

\index{$p^*_{\IQPp}$, optimal value of \IQP}
\index{optimal value of \IQP, $p^*_{\IQPp}$}

We are looking to solve the following integer quadratic program over the indicator vector $x$:
\begin{equation}
\label{prob:original form}
\begin{array}{rl}
p^*_{\IQPp} := 
\min\limits_x \  & \sum\limits_{u,v} E_{uv} x_{u} x_{v} \\
\text{s.t.} \ & \sum\limits_{u \in \cV _{k}} x_{u} = 1, \ \ k= 1, \dots, p \\
& x = 
 \begin{bmatrix}
 w_{1}^{T}  & w_{2}^{T} & \dots & w_{p}^{T} 
 \end{bmatrix}^T   \\
& w_{i} \in \{ 0,1 \}^{m_{i}}, \ i = 1,\dots, p.
\end{array}
\end{equation}
The constrains in \eqref{prob:original form} forces that exactly one element  of $v_{i}$ is set to be $1$, consequently modelling that exactly one rotamer is chosen for each rotamer set $\cV_i$.
We construct the block diagonal matrix
\index{$A$}
\index{$n_0$, total number of rotamers}
\index{total number of rotamers, $n_0$}
\begin{equation}
\label{eq:def_A}
A = 
\textdef{$\BlkDiag$}(  \bar{e}_{m_{1}}^{T} , \bar{e}_{m_{2}}^{T} ,\cdots
,\bar{e}_{m_{p}}^{T} )
\in \mathbb{R}^{p\times n_0} .
\end{equation}
We then use $Ax=b$ to work with a concise representation of the first equality constraint in~\eqref{prob:original form}.
Finally, we obtain the following representation of the \SCP problem:
\begin{equation}
\label{problem:IQP}
(\textdef{\IQPp}) \qquad
\begin{array}{ccl}
p^*_{\IQPp} \ = &  \min\limits_x \ & x^{T}Ex\\
&\text{ s.t. }  & Ax = \bar{e}_{p}  \\
&& x  \in \{ 0,1\}^{n_{0}} . \\
\end{array}
\end{equation}

\subsection{\SDP Relaxation}
\label{sec:SDP_Relaxation}

\index{$\hat{E}$} 
\index{$E_{00}$}
\index{$e_0$, the first unit vector}
\index{the first unit vector, $e_0$}

The problem \eqref{problem:IQP} is NP-hard and hence we resort to a relaxation.
We define 
\[
\hat{E} := \BlkDiag(0,E)  \in \mathbb{S}^{n_0+1}, \ \ 
E_{00} := e_0 e_0^T \in \mathbb{S}^{n_0+1},
\]
where $e_0$ is the first unit vector.
In this section we aim to obtain the following \SDP relaxation to the discrete optimization problem~\eqref{problem:IQP}:
\begin{equation}
(\textdef{\SDPp}) \qquad
\label{eq:SDP_FR_orig}
\begin{array}{rl}
p^*_{\SDPp} \ :=  \
\min\limits_{R,Y} & \trace(\hat{E}Y) \\
& G_{\hat \cJ}(Y) = E_{00} \\
& Y = VRV^T \\
& R \in \mathbb{S}^{n_0+1-p}_{+} ,  \\
\end{array}
\end{equation}
where $G_{\hat \cJ}( \cdot )$ and $V$ are explained in \Cref{sec:GangsterConst} and \Cref{sec:FRconst}, respectively.  
A variant of the relaxation \eqref{eq:SDP_FR_orig} is proposed
by~\cite{ForbesVrisWolk:11} via Lagrangian relaxation to \eqref{problem:IQP}.
Here we present a simpler derivation of the model \eqref{eq:SDP_FR_orig} via 
a simple direct lifting.

The first step for deriving the \SDP relaxation \eqref{eq:SDP_FR_orig} is to lift the variable dimension.
Given~$x\in \R^{n_0}$, we lift to symmetric matrix space
using the rank-one \emph{lifted matrix} 
\index{$Y_x$}
\[
Y_x 
:= \begin{bmatrix} 1\\x  \end{bmatrix}
\begin{bmatrix} 1\\x  \end{bmatrix}^T
= \begin{bmatrix} 1 & x^T \\x & xx^T  \end{bmatrix} \in \mathbb{S}^{n_0+1}.
\]
For the \SDP relaxation, we index the rows and columns \emph{starting from $0$}, i.e., the row and column indices are $\{0,1,\ldots, n_0\}$.
This lifting allows for an alternative representation of the objective function 
\[
x^TEx 
= \left\<  
\begin{bmatrix} 0 & 0 \\ 0 & E \end{bmatrix} , 
 \begin{bmatrix} 1\\x  \end{bmatrix}
\begin{bmatrix} 1\\x  \end{bmatrix}^T
\right\>
= \left\< \hat{E}, Y_x \right\> . 
\]
In the remaining of this section we show 
how this lifting process gives rise to the constraints of the model \eqref{eq:SDP_FR_orig}:
\begin{enumerate}
\item
the  linear (gangster) constraint $G_{\hat \cJ}(Y) = E_{00}$ (\Cref{sec:GangsterConst});
\item
$Y = VRV^T$, where $R \in \mathbb{S}^{n_0+1-p}_{+}$ (\Cref{sec:FRconst}).
\end{enumerate}

\subsubsection{Gangster Constraint $G_{\hat \cJ}(Y) = E_{00}$}
\label{sec:GangsterConst}

Given a matrix $W \in \mathbb{S}^{n_0}$, we define the set of indices  
\index{$\cJ$}
\[
\hspace{-0.3cm}
\cJ : = 
\left\{
\left( \sum_{i=1}^j{m}_{i-1}+ k  ,\ \sum_{i=1}^j{m}_{i-1} + \ell \right) :
j \in \{1,\ldots,p-1\}, \  k,\ell \in \{ 2,\ldots,m_i-1 \}, \ k \ne \ell
\right\} .
\]
Here, $m_i$ is the cardinality of rotamer set $\cV _i$, 
and $m_0 = 0$.
In other words, 
$\cJ$ is the set of off-diagonal indices of the $m_i$-by-$m_i$ diagonal 
\begin{figure}[h]
\centering
\includegraphics[height=3.7cm]{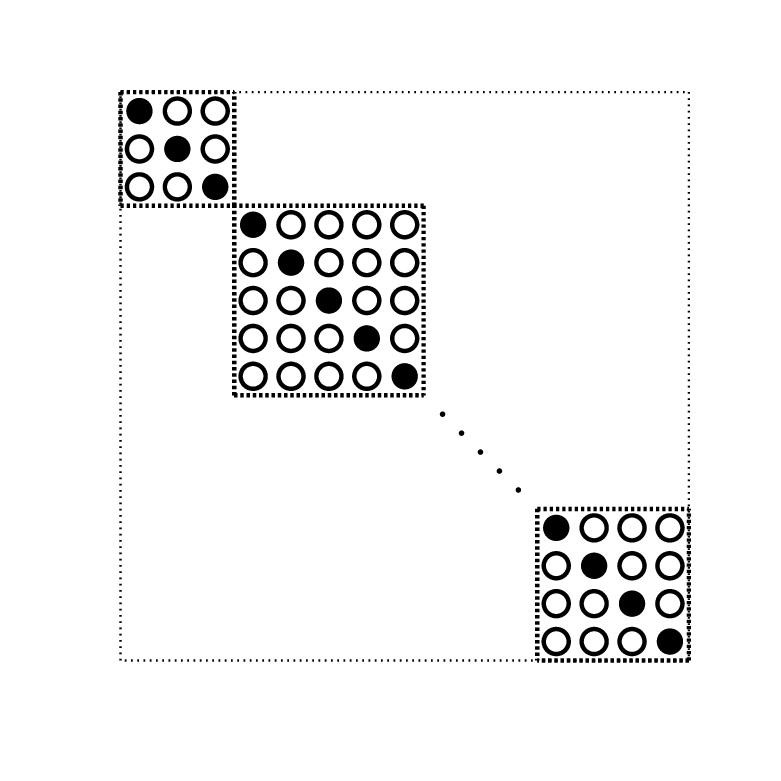}
\vspace{-0.5cm}
\caption{An illustration of the index set $\cJ$ of $0$'s; members of $\cJ$ correspond to the off-diagonal elements of diagonal blocks of $W$ indicated by the symbol \begin{large}$\circ$\end{large}.}
\label{fig:Gangsterinds}
\end{figure}
blocks of $W\in \mathbb{S}^{n_0}$; see \Cref{fig:Gangsterinds} for a visual illustration of the positioning of these indices.
Note that these indices correspond to exactly 
\[
W_{uv}=x_ux_v=0, u\neq v, u,v\in \cV_i,
\]
i.e.,~the constraint on any two distinct rotamers in the same rotamer set cannot be chosen. 

With the above set of indices, we define the mapping
\index{$G_{\cJ}$, gangster operator}
\index{gangster operator, $G_{\cJ}$}
\[
G_{\cJ} : \mathbb{S}^{n_0} \to \R^{|\cJ|}  \ \ \text{ by }
G_{\cJ}(W) = \left( W_{ij} \right)_{ij\in \cJ} .
\]
By abuse of notation, we also view the mapping $G_\cJ$ as an operator from $\mathbb{S}^{n_0} $ to $\mathbb{S}^{n_0}$ to mean 
\[
G_{\cJ} : \mathbb{S}^{n_0} \to \mathbb{S}^{n_0}  ,  \ \  
(G_{\cJ}(W))_{i,j} = 
\bigg\{ 
\begin{array}{cl}
W_{i,j} & \text{if }(i,j) \text{ or } (j,i)\in \cJ, \\
0       & \text{otherwise}.
\end{array}
\]
The map $G_{\cJ}$ can also be viewed as the operator on $\Sc^{n_0}$
defined by $G_{\cJ}(W) = (A^T A - I) \circ W $ with $A$ defined in \eqref{eq:def_A}.
Recall $\circ$ is the element-wise matrix product.
In plain words, $G_{\cJ}(W)$ is the projection that
chooses elements of $W$ corresponding to the index set $\cJ$. 
The constraint $G_{\cJ}(W)=0$ is often called 
the \textdef{gangster constraint} and it is due to the fact that 
elements of $W$ associated with $\cJ$ are set to be zero~(shoots holes in the matrix).

\index{$\hat{\cJ}$}
We now define the set of pairs of indices 
\[
\hat{\cJ} := \{(0,0)\} \cup \cJ \subset \{0,1,\ldots, n_0\}\times \{0,1,\ldots, n_0\}
\]
to directly work with the lifted variable in $\mathbb{S}^{n_0+1}$.
We define the analogous mapping $G_{\hat{\cJ}} $ with $\hat{\cJ}$:
\[
G_{\hat{\cJ}} : \mathbb{S}^{n_0+1} \to \R^{|\hat{\cJ}|}  \ \ \text{ by } \ G_{\hat{\cJ}}(Y) = \left( Y_{ij} \right)_{ij\in \hat{\cJ}} .
\]
This yields the \textdef{gangster constraint} in projection and operator
equivalent forms, respectively,
\[
G_{\hat{\cJ}} (Y) = e_0 \in \R^{1+|\cJ|} , \quad G_{\hat{\cJ}} (Y) = E_{00}.
\]

\index{$G_{\hat{\cJ}} $, gangster operator}
\index{gangster operator, $G_{\hat{\cJ}} $}

\subsubsection{Facial Reduction}
\label{sec:FRconst}

We now derive the constraint $Y = VRV^T$ and $ R \in 
\mathbb{S}^{n_0+1-p}_{+}$. 
Let $x$ be a feasible solution to~\eqref{problem:IQP} and we observe the following implications:
\index{$K$}
\[
\begin{array}{rl}
Ax = \bar{e}_p 
& \implies \begin{bmatrix} 1\\x \end{bmatrix}^T \begin{bmatrix} - \bar{e}_{p}^T\\ A^T\end{bmatrix} = 0 \\
& \implies
\begin{bmatrix} 1\\x \end{bmatrix} \begin{bmatrix} 1\\x \end{bmatrix}^T \begin{bmatrix} - \bar{e}_{p}^T\\A^T\end{bmatrix} \begin{bmatrix} -\bar{e}_{p}^T\\A^T\end{bmatrix}^T  = 0 \\
& \implies \left\< 
\underbrace{
\begin{bmatrix} 1\\x \end{bmatrix} \begin{bmatrix} 1\\x \end{bmatrix}^T }_{=Y_x} \ , 
\underbrace{
\begin{bmatrix} - \bar{e}_{p}^T\\A^T\end{bmatrix} \begin{bmatrix} -\bar{e}_p^T \\ A^T\end{bmatrix}^T }_{=:K}  \right\rangle =0.
\end{array}
\]
Since both arguments in the last inner product are positive semidefinite, 
we obtain the useful property:
\begin{equation}
\label{eq:FRimplications}
\<K,Y_x\> = 0 
\implies KY_x =0 
\implies \range(Y_x) \subseteq \Null(K).
\end{equation}
In other words, $\Null(K)$ captures the range that the feasible points can have.

We now exploit the property \eqref{eq:FRimplications} to restrict the range of the variable. 
We find a full-column rank matrix 
$V \in \R^{(n_0+1)\times (n_0+1-p) }$ such that 
\[
\range(V)
= \Null(K)
= \Null\left(
 \begin{bmatrix} -\bar{e}_p^T \\ A^T\end{bmatrix}^T \right)  .
\]
For our purposes we choose $V$ with normalized columns.
Since $A$ is full-row rank, we get that $\rank(K) = p$. 
Finally, we can capture any feasible $Y_x$ using $V$:
\[
Y_x \in V  \mathbb{S}^{n_0+1-p}_{+} V^T.
\]
This is the well-known \textdef{facial reduction} technique, see
e.g.,~\cite{DrusWolk:16}.
The matrix $K$ functions as an \textdef{exposing vector} for the feasible set.
The matrix $V$ is known as a facial range vector.

The remaining step for the \SDP relaxation is simple. 
We note that $\rank( Y_x)=1$ and this leaves the feasible region nonconvex. 
We discard the rank restriction on the variable $Y_x$ to work with a convex feasible region and the variable of the form
\[
Y = VRV^T \text{ where }  R \in \mathbb{S}^{n_0+1-p}_{+} . 
\]
This completes the derivation of the relaxation in~\eqref{eq:SDP_FR_orig}.
It is known that there is a $\hat{R} \in \mathbb{S}^{n_0+1-p}_{++}$  feasible to \eqref{eq:SDP_FR_orig}; see \cite{ForbesVrisWolk:11}.

\subsection{\DNN Relaxation}
\label{sec:DNN_Relaxation}

We continue with the \SDP relaxation derived in \Cref{sec:SDP_Relaxation} to 
complete our relaxation by adding additional constraints to \eqref{eq:SDP_FR_orig}.
In \Cref{thm:arrow_redundancy} below, we obtain two additional properties of the model \eqref{eq:SDP_FR_orig}.
\begin{thm}
\label{thm:arrow_redundancy}
Suppose that $(R,Y)$ are feasible to \eqref{eq:SDP_FR_orig}.
Then the following hold.
\begin{enumerate}
\item \label{item:arrowEqual}
The first column of $Y$ is equal to the diagonal of $Y$.
\item \label{item:tracd1+p}
$\tr (R) = 1+p$.
\end{enumerate}
\end{thm}

\begin{proof}
We recall that $\range(V) = \Null (K) = \Null \left( \begin{bmatrix} -\bar{e}_p & A \end{bmatrix} \right)$.
Hence we have
\begin{equation}
\label{eq:0RVT=0proof}
\begin{bmatrix} -\bar{e}_p & A \end{bmatrix} Y
= \begin{bmatrix} -\bar{e}_p & A \end{bmatrix} VRV^T = 0 RV^T = 0.
\end{equation}
We then exploit the structure of 
$\begin{bmatrix} -\bar{e}_p & A \end{bmatrix} Y$.
We first partition $Y$ as follows:
\begin{equation}
\label{eq:block_notation}
Y = 
\begin{bmatrix}
1 & Y_{10}^T & Y_{20}^T & \cdots & Y_{p0}^T  \\
Y_{10} & Y_{11}   & Y_{12}   & \cdots & Y_{1p} \\
\vdots & \vdots   & \vdots   &  \vdots & \vdots \\  
Y_{p0} & Y_{p1}   & Y_{p2}   & \cdots & Y_{pp} \\
\end{bmatrix}  \in \Sc^{n_0+1} ,
\end{equation}
where $Y_{ii} \in \Sc^{m_i}$, $Y_{ij} \in \R^{m_i \times m_j}$, $Y_{i0} \in \R^{m_i}, \ \forall i,j \in [p]$.
We use $ Y_{ij}^{\col \ell}$ to denote the $\ell$-th column of the 
$(i,j)$-th block of $Y$ and $Y_{i0,\ell}$ to denote the 
$\ell$-th coordinate of the vector $Y_{i0} \in \R^{m_i}$.
\index{$ Y_{ij}^{\col \ell}$}
Then expanding $\begin{bmatrix} -\bar{e}_p & A \end{bmatrix} Y$ with the block representation \eqref{eq:block_notation} yields
\[
\begin{bmatrix} -\bar{e}_p & A \end{bmatrix} Y
=
\begin{bmatrix} a_0 & A_1 & \cdots & A_p \end{bmatrix} \in \R^{p\times (n_0+1)},
\]
where 
\begin{equation}
\label{defof_a_0proof}
a_0 = 
\begin{bmatrix}
-1 + \bar{e}^T_{m_1} Y_{10} \\
-1 + \bar{e}^T_{m_1} Y_{20} \\
\vdots \\
-1 + \bar{e}^T _{m_p}Y_{p0} \\
\end{bmatrix}  \in \R^{p} , 
\end{equation}
and, for each $i\in [p]$,
\[
A_i = 
\begin{bmatrix}
-Y_{i0,1} + \bar{e}^T_{m_1} Y_{1i}^{\col 1} & -Y_{i0,2} + \bar{e}^T_{m_1} Y_{1i}^{\col 2} 
& \cdots & -Y_{i0,m_i} + \bar{e}^T_{m_1} Y_{1i}^{\col m_i}  \\
\vdots &\vdots &\ddots &\vdots  \\ 
-Y_{i0,1} + \bar{e}^T_{m_j} Y_{ii}^{\col 1} & -Y_{i0,2} + \bar{e}^T_{m_j} Y_{ii}^{\col 2} 
& \cdots & -Y_{i0,m_i} + \bar{e}^T_{m_j} Y_{ii}^{\col m_i}  \\
\vdots &\vdots &\ddots &\vdots  \\ 
-Y_{i0,1} + \bar{e}^T_{m_p} Y_{1p}^{\col 1} & -Y_{i0,2} + \bar{e}^T_{m_p} Y_{1p}^{\col 2} 
& \cdots & -Y_{i0,m_i} + \bar{e}^T_{m_p} Y_{1p}^{\col m_i}  \\
\end{bmatrix}  \in \R^{ p \times m_i}.
\]
By \eqref{eq:0RVT=0proof}, we have $A_i =0, \ \forall i\in [p]$.
Thus, for each $i\in [p]$, the $i$-th row of $A_i$ yields
\[
Y_{i0,\ell} = \bar{e}^T_{m_i} Y_{ii}^{\col \ell}, \  
\ell \in [m_i].
\]
Since $G_{\hat{\cJ}}(Y) = E_{00}$ holds, we see that 
\[
\diag(Y_{ii}) = Y_{i0}, \ \forall i \in [p].
\]
Therefore, we conclude that the first column and the diagonal of $Y$ are identical.

We now show that $\tr (R) = 1+p$. By \eqref{eq:0RVT=0proof}, we have the vector $a_0$ from \eqref{defof_a_0proof} is $0$. Thus, $1=\bar{e}^T_{m_i}$, for all $i=1\ldots,p$. By \Cref{item:arrowEqual}, we obtain
\[
\bar{e}^T_{m_i} Y_{i0} = 1, \ \forall i \in [p].
\]
Since $\diag(Y_{ii}) = Y_{i0}, \ \forall i \in [p]$, we must have that $\trace(Y_{ii}) = 1$, $\forall i \in [p]$.
Hence $Y=VRV^T$ gives 
\[
1+p = \tr (Y) = \tr(VRV^T) = \tr (R),
\]
where the last equality holds since $V^TV = I$.
\end{proof}

\Cref{item:arrowEqual} of \Cref{thm:arrow_redundancy} is known in the literature; see \cite{ForbesVrisWolk:11}.
This property is discovered by using the Lagrangian dual. 
Here, we displayed an alternative derivation that exploits the steps during the direct lifting.

We recall that the original model \eqref{problem:IQP} has the binary constraint on its variable $x$.
We also recall that the direct lifting yields a variable of the form 
$\begin{bmatrix} 1 & x^T \\x & xx^T  \end{bmatrix} \in \Sc^{n_0+1}$.
Hence, we strengthen our model by including the constraint $ Y_{i,j} \in [0,1], \ \forall i,j$.

\index{$\cY$}
\index{$\cR$}
\index{optimal value of \DNN relaxation, $p^*_{\DNNp}$}
\index{$p^*_{\DNNp}$, optimal value of \DNN relaxation}

We define the sets
\[
\begin{array}{rcl}
\cR &:=& \{ R \in \Sc^{n_0+1-p} : R\succeq 0, \ \tr(R) = p+1 \} , \\
\cY &:=& \{ Y \in \Sc^{n_0+1} :  G_{\hat{\cJ}}(Y) = E_{00}, \ 0\le Y \le 1 \} .
\end{array}
\]
By including additional constraints $\tr (R) = 1+p$ and $0\le Y \le 1$ to the \SDP relaxation \eqref{eq:SDP_FR_orig}, we complete our model, the \DNN  relaxation to \eqref{problem:IQP}:
\begin{equation}
\label{eq:SDP_FR}
(\textdef{\DNNp}) \qquad
\begin{array}{rlc}
p^*_{\DNNp} \ := \ \min\limits_{R,Y} & \tr (\hat{E}Y) \\
& Y = VRV^T   \\
& R \in \cR \\
& Y\in \cY . \\
\end{array}
\end{equation}

We remark that both (\DNNp) and (\SDPp) are relaxations to (\IQPp), but 
(\DNNp) is a strengthened model than (\SDPp), i.e.,
\[
p^*_{\SDPp} \le p^*_{\DNNp} \le p^*_{\IQPp} .
\]
We also remark that there are redundant constraints in the model \eqref{eq:SDP_FR}. 
These (implicit) redundant constraints result in numerical instabilities when they are not treated carefully.
In \Cref{sec:DNN_Relaxation} below, we use the splitting method to distribute constraints to two different subproblems. 
We benefit from the usage of the splitting method in two distinct ways; 
we handle the numerically difficult problem into two separate easier subproblems; and
we avoid the numerical instabilities that arise from the redundant constraints.

We demonstrate the strength of (\DNNp) in \Cref{sec:AtightRelaxation}. 
The \DNN relaxation has a linear objective with an \emph{onto} linear equality
constraint, and compact, convex, feasible set constraints.
The first-order optimality conditions for \eqref{eq:SDP_FR} are
\begin{equation}
\label{eq:optcondition}
\begin{array}{lll}
	0&\in  -V^TZV+\NN_{\cR}(R),
            &\quad \text{(dual feasibility with respect to $R$)} \\
	0&\in  \hat{E}+Z+\cN_{\cY}(Y),
            &\quad \text{(dual feasibility with respect to $Y$)} \\
	Y & =   \widehat{V}R\widehat{V}^T,
              \quad R \in \cR, \,  Y \in \cY,
            & \quad \text{(primal feasibility)}
\end{array}
\end{equation}
where
$\cN_{\cR}(R), \cN_{\cY}(Y)$ are the normal cones and $Z$ is a Lagrange multiplier associated with the constraint $Y=VRV^T$.
\Cref{thm:EQzero} below states that some elements of the dual optimal multiplier $Z^*$ are known in advance.
\begin{theorem}
\label{thm:EQzero}
Let $(R^*,Y^*)$ be an optimal pair for~\eqref{eq:SDP_FR}, and let
\[
\textdef{$\ZZ_A$} :=\left\{Z\in\Sc^{n_0+1}: Z_{i,i} = -(\hat{E})_{i,i}, 
\ Z_{0,i}=Z_{i,0}= -(\hat{E})_{0,i} , \ i=1,\ldots ,n_0 \right\}.
\]
Then there exists $Z^*\in\ZZ_A$ such that $(R^*,Y^*,Z^*)$ 
solves \cref{eq:optcondition}.	
\end{theorem}
\begin{proof}
The proof uses the optimality conditions \eqref{eq:optcondition} and  \Cref{thm:arrow_redundancy}.
The proof can be found in~\cite[Theorem 2.11]{GHILW:20}.
\end{proof}

\section{The Algorithm}
\label{sec:Algorithm}

In this section we present the algorithm for solving the \DNN relaxation  \eqref{eq:SDP_FR}.
For $\beta >0$, we define the augmented Lagrangian $\cL_A$ of the model \eqref{eq:SDP_FR}:
\index{$\cL_A(R,Y,Z)$, augmented Lagrangian}
\index{augmented Lagrangian, $\cL_A(R,Y,Z)$}
\begin{equation}
\label{eq:aug_Lagrangian}
\cL_A(R,Y,Z) :=
\<\hat{E},Y\> + \<Z,Y-VRV^T\> + \frac{\beta}{2}\|Y-VRV^T\|_F^2.
\end{equation}
We define the projection operator $\cP_{\cZ_0}(Z)$ onto the set 
\[
\cZ_0 = \left\{Z\in\Sc^{n_0+1}: 
Z_{i,i} = \ Z_{0,i}=Z_{i,0}= 0 , \ i=1,\ldots ,n_0 \right\} .
\]
In other words, the projection operator $\cP_{\cZ_0}(Z)$ sets the first column, first row and and the diagonal elements of $Z$ to be $0$, except for the $(0,0)$-th entry.

\index{$\cP_{\mathcal{Z}_0}(Z)$}
\index{\PRSMp, Peaceman-Rachford splitting method}
We use the \textdef{restricted dual Peaceman-Rachford splitting method, \rPRSM} (\Cref{algo:PRSM_algorithm}), a variation of the strictly contractive
Peaceman-Rachford splitting method to solve the model~\eqref{eq:SDP_FR}.
\begin{algorithm}[h!]  
\caption{\rPRSM \cite{GHILW:20} for solving \eqref{eq:SDP_FR} }
\label{algo:PRSM_algorithm} 
\begin{algorithmic}
\STATE \textbf{Initialize:} $Y^0 \in \Sc^{n_0+1}$, $Z^0\in \mathcal{Z}_A$, $\beta\in (0,\infty), \gamma \in (0,1)$
\WHILE {termination criteria are not met}
\STATE $R^{k+1} = \argmin\limits_{R\in \cR} \LL_A (R,Y^k,Z^k)$
\STATE $Z^{k+\frac{1}{2}} = Z^{k} + \gamma  \beta \cdot \cP_{\mathcal{Z}_0}\left( Y^{k}- V R^{k+1} V^T \right)$ 
\STATE $Y^{k+1} = \argmin\limits_{Y\in \mathcal{Y} } \LL_A (R^{k+1},Y,Z^{k+\frac{1}{2}}) $
\STATE $Z^{k+1} = Z^{k+\frac{1}{2}} + \gamma  \beta \cdot \cP_{\mathcal{Z}_0}\left( Y^{k+1}- V R^{k+1} V^T \right)$ 
\ENDWHILE
\end{algorithmic}
\end{algorithm}
We note that the ordinary \PRSM updates the dual multipliers without the projection operator $\cP_{\mathcal{Z}_0}$.
The projection on the dual multiplier~$Z$ is motivated from an endeavour to have a better dual multiplier at each iteration. 
We recall that some of the elements of the optimal dual multipliers are known by \Cref{thm:EQzero}. 
The algorithm fixes these known elements to be the optimal elements at every iteration.
We leave the details of the convergence proof of \rPRSM scheme to~\cite[Theorem 3.2]{GHILW:20}.

\begin{remark}
The model \eqref{eq:SDP_FR} can be solved by using a standard \SDP solver.
The nonnegativity of each element of $Y$ is considered using cutting planes in \cite{ForbesVrisWolk:11}.
However, this approach becomes more computationally challenging as the number of cutting planes increases. 
Splitting methods engage the polyhedral constraints $0 \le Y \le 1$ in an economic manner.
We incorporate the positive semidefinite constraint and the nonnegativity constraint very efficiently.
We deal with the positive semidefinite and trace
constraint in the $R$-subproblem,
and then deal with the interval and gangster constraints in the $Y$-subproblem.
\end{remark}

\subsection{Update Formulae}
\label{sec:Update_Formulae}

In this section we present the formulae for the $R$ and $Y$ updates in \Cref{algo:PRSM_algorithm}. 
The update rules are discussed in~\cite{GHILW:20}. We include the formulae for completeness.

\subsubsection{$R$-Update}
\label{sec:R_update}

In this section we present the update rule for the $R$-subproblem.
The formula for the $R$-subproblem, with $\cL_A$ defined in \eqref{eq:aug_Lagrangian}, is as follows: 
\[
\begin{array}{rcl}
R^{k+1} 
& = & \argmin\limits_{R\in \cR} \cL_A(R,Y^{k},Z^{k} ) \\
& = & \argmin\limits_{R\in \cR} \left\| Y^{k} - VRV^T + \frac{1}{\beta}Z^{k} \right\|_F^2 \\
& = & \argmin\limits_{R\in \cR} \left\| R - V^T (Y^{k}+\frac{1}{\beta}Z^{k} )V \right\|_F^2 \\
& = & \cP_{\cR} \left( V^T \left( Y^{k} + \frac{1}{\beta}Z^{k} \right) V \right) \\
& = & U \ \Diag\left( \cP_{\Delta_{p+1}} (d) \right) \ U^T ,
\end{array}
\]
where the second equality holds by completing the square; the third
equality holds due to $V^TV = I$; and the last equality follows from the eigenvalue decomposition
\[
V^T \left( Y^{k} + \frac{1}{\beta}Z^{k} \right) V 
= U \Diag(d) U^T , 
\]
and $ \cP_{\Delta_{p+1}} (\cdot ) $ is the projection operator onto the simplex $\Delta_{p+1} = \{ z\in \R^{n_0+1-p} : \bar{e}^T z = 1+p \}$.

\subsubsection{$Y$-Update}
\label{sec:Y_update}
The update rule for $Y$ is as follows:
\begin{equation} 
\label{Y_update_complete_sqaure}
\begin{array}{rcl}
Y^{k+1} 
& = & \argmin\limits_{Y\in \cY} \cL_A(R^{k+1},Y,Z^{k+\frac{1}{2}} ) \\
& = & \argmin\limits_{Y\in \cY} \left\| Y - \left( V R^{k+1} V^T - \frac{1}{\beta} (\hat{E}+Z^{k+\frac{1}{2}} ) \right)  \right\|_F^2 \\
& = & \cP_{\text{box}} \left( G_{\hat{\cJ}^c } \left( V R^{k+1} V^T - \frac{1}{\beta} (\hat{E}+Z^{k+\frac{1}{2}} ) \right)  \right),
\end{array}
\end{equation}
where $\cP_{\text{box}}$ is the projection onto the polyhedral set $\{Y
\in \Sc^{n_0+1} : 0\le Y \le 1\}$. 

\index{$\cP_{\text{box}}$}

\subsection{Bounding}
\label{sec:Bounding}

In this section we present some strategies for computing lower and upper bounds to (\IQPp).

\subsubsection{Lower Bounds from Lagrange Relaxation}
\label{sec:Lower_Bound}

We now discuss a strategy for computing a valid lower bound to $p^*_{\IQPp}$.
Exact solutions of the \DNN relaxation~\eqref{eq:SDP_FR} provide lower bounds to (\IQPp).
However, we often terminate algorithms when the stopping criteria are met for a pre-defined tolerance and we never set the tolerance to be exactly $0$ in practice. 
A near optimal point $\tilde{Y}$ can result in 
\[
p^*_{\DNN} \le \<\hat{E},\tilde{Y}\> \ \text{ and } \
p^*_{\IQP} < \<\hat{E},\tilde{Y}\> 
\]
and produce an invalid lower bound to $p^*_\IQPp$.
Hence, we provide a method for computing a \emph{valid lower bound} to (\IQPp) for avoiding this issue.

We follow the approaches in \cite{GHILW:20,OliveiraWolkXu:15,Eckstein:2020} and obtain lower bounds via the dual to the \DNN
relaxation in \eqref{eq:SDP_FR}.
We define the dual functional $g:\Sc^{n_0+1} \to \R$ by
\index{$g$, dual functional}
\index{dual functional, $g$}
\[
g(Z) := \min\limits_{R\in \cR, Y\in \cY }\<\hat{E},Y\rangle + \< Z,Y-VRV^T\> .
\]
Let $\bar{Z}\in \bS^{n_0+1}$ be given. 
W note that
\[
\begin{array}{rl}
\min\limits_{R\in \cR, Y\in \cY } \<\hat{E},Y\rangle + \< \bar{Z} , Y-VRV^T\> 
& = \min\limits_{Y\in \mathcal{Y} } \<\hat{E}+\bar{Z},Y\>  + \min\limits_{R\in \RR} \<-V^T \bar{Z} V,R\>  \\
& = \min\limits_{Y\in \mathcal{Y} } \<\hat{E}+ \bar{Z} ,Y\> - (p+1) \lambda_{\max}(V^T \bar{Z} V),
\end{array}
\]
where \textdef{$\lambda_{\max}$} is the maximum eigenvalue function.
Hence we compute a valid lower bound to the optimal value $p^*_{\DNN}$ of the model \eqref{eq:SDP_FR} by using weak duality:
\begin{equation*}
\label{eq:lbd_computation}
p^*_{\DNN} = \max\limits_Z  g(Z) 
\ge g(Z) 
= \min\limits_{Y\in \mathcal{Y} } \<\hat{E}+Z,Y\> - (p+1) \lambda_{\max}(V^TZV),
\end{equation*}
where the first equality holds since the constraint qualification holds for the model \eqref{eq:SDP_FR}.
We note that the computation for $\min\limits_{Y\in \mathcal{Y} } \<\hat{E}+Z,Y \rangle$ is inexpensive.



\subsubsection{Upper Bounds from Nearest Binary Feasible Solutions}
\label{sec:Upper_Bound}

In this section we discuss two strategies for computing upper bounds to 
the \SCP problem.  These strategies are derived from those presented 
in~\cite{ForbesVrisWolk:11} and we include them here for completeness.
We obtain upper bounds by finding feasible solutions to the original
integer model in \eqref{problem:IQP}.
Let $(\Rout,\Yout,\Zout)$ be the output of the algorithm.
\begin{enumerate}
\item
Let $\xapprox \in \R^{n_0}$ be the second through to the last 
elements of the first column of $\Yout$. 
Note that $0\leq \xapprox \leq 1$.
Then the nearest feasible solution to (\IQPp) from $\xapprox$ can be found by solving the following projection:
\begin{equation}
\label{eq:projection_onto_fea}
\min_x \left\{ \|x - \xapprox\|^2 \ : \ Ax = \bar{e}_p, \ x\in \{0,1\}^{n_0}  \right\}.
\end{equation}
It is shown in~\cite{ForbesVrisWolk:11} that 
solving \eqref{eq:projection_onto_fea} is equivalent to solving the
following \emph{linear program}:
\begin{equation}
\label{eq:nearest_sol_LP}
\min_x \left\{ \<x , \xapprox \> \ : \ Ax = \bar{e}_p, \ x \ge 0  \right\}.
\end{equation}
\item
We now let $\xapprox$ be the second through to the last elements of the most dominant eigenvector of $\Yout$.
Note that we again have $0\leq \xapprox \leq 1$, by the Perron-Frobenius
theorem. We again obtain the nearest feasible solution to $\xapprox$ by
solving \eqref{eq:nearest_sol_LP}.
\end{enumerate}

\begin{remark}
In fact, solving \eqref{eq:nearest_sol_LP} does not require using any LP software; we can obtain the optimal solution for \eqref{eq:nearest_sol_LP} as follows.
We partition $\xapprox$ into $p$ subvectors of sizes $m_i=|\cV_i|$, for $i=1,\ldots,p$.
Let $x^{i}\in \R^{m_i}$ be the subvector of $\xapprox$ associated with $i$-th rotamer set $\cV_i$, i.e., 
$\xapprox = [ x^{1};x^{2}; \ldots ;x^{p}]$. 
We define $\hat{x}^i \in \R^{m_i}$ as follows:
\[
\hat{x}^i_j 
= \left\{ \begin{array}{ll}
1 , & \text{if } x^i_j = \max\limits_{\ell \in [m_i] } \left\{x^i_\ell \right\}   \\
0, & \text{otherwise}.\\
\end{array} \right.
\]
If there is subvector $ \hat{x}^{i}$ with more than one $1$ in its components, we pick only one $1$ and set the remaining to be $0$.
We then form $\hat{x} = [\hat{x}^1;\hat{x}^2; \ldots;\hat{x}^p ] \in \R^{n_0}$.
It is clear that $\hat{x}$ is feasible for~\eqref{prob:original form}.
We use $\hat{x}^T E \hat{x}$ as an upper bound to the \SCP problem.
\end{remark}

\section{Numerical Experiments with Real-World Data}
\label{sec:Numerical_Experiments}

We present the numerical experiments for \Cref{algo:PRSM_algorithm}.
This section is organized as follows.
In \Cref{sec:stopping_Criteria} we present the parameter settings and stopping criteria. 
In \Cref{sec:Data} we explain how we process the data from the Protein Data Bank (PDB) to obtain the energy matrix $E$.
In \Cref{sec:Experiments} we finally present the numerical results using
\rPRSM and show that we provably solve many instances to optimality.
We use the bounding strategies presented in \Cref{sec:Bounding} 
to prove optimality.


\index{Protein Data Bank, PDB}
\index{PDB, Protein Data Bank}

\subsection{Stopping Criteria and Parameter Settings}
\label{sec:stopping_Criteria}

\paragraph{Stopping Criteria}

We terminate \rPRSM when either of the following conditions is satisfied. 
\begin{enumerate}
\item 
Maximum number of iterations, denoted by ``maxiter'' is achieved. 

\item 
For given tolerance $\epsilon$, the following bound on the primal and 
dual residuals holds for \textdef{$s_t$} sequential times:
\[
\max \left\{ 
\frac{\|Y^k- VR^k V^T\|_F}{\|Y^k\|_F}  \
\beta \|Y^{k}-Y^{k-1}\|_F 
\right\} <\epsilon.
\]

\item 
Let $\{ l_1,\ldots, l_k \}$ and  $\{ u_1,\ldots,u_k \}$
be sequences of lower and upper bounds discussed in \Cref{sec:Lower_Bound} and \Cref{sec:Upper_Bound}, respectively.
Any of the lower bounds achieve the best upper bound, i.e., 
\[
\min \{ l_1,\ldots, l_k \} \ge \max \{ u_1,\ldots,u_k \} .
\]

\end{enumerate}

\paragraph{Parameter Settings}

We use the following parameters related to the implementation of \Cref{algo:PRSM_algorithm}:
\[
\beta = \max\{ \lfloor 0.5*n_0/p \rfloor,1\},  \ \
\gamma = 0.99 .
\]
The parameters related to stopping criteria are:
\[
\text{maxiter} = p(n_0+1) + 10^4 , \ \
\epsilon = 10^{-10}, \ \ 
s_t = 100.
\]
For the initial iterates for \rPRSMp, we use 
\[
Y^0= 0 , \ Z^0 = \cP_{\cZ_A} (Y^0) .
\]

\subsection{Energy Matrix Computation}
\label{sec:Data}

In this section we briefly describe the process for acquiring the energy matrix $E$. 
Our implementation relies on the usage of a Python script executing as an extension of the UCSF Chimera\footnote{The UCSF Chimera software can be found in \url{https://www.cgl.ucsf.edu/chimera/download.html}.} application. A detailed implementation can be found in \cite[Chapter 7]{BurkowskiForbesJ2015CaVT}.
We used protein data files from the PDB to obtain the coordinates of all atoms in the protein.  To get the energy values required by the  algorithm, the native side chain conformations were replaced by rotamers extracted from a rotamer library provided by the Dunbrack Laboratory \cite{Dunbrack341254}.

Some approaches use an energy evaluation based on a piece-wise linear approximation of the Lennard-Jones potential formula
(e.g., \cite{XuJinbo2006Faaa,CanutescuAdrianA2003Agaf}).
Here, we used the Lennard-Jones potential formula, which provides a more accurate energy value computation. 
In brief, the Lennard-Jones potential formula engages the Euclidean distance between a pair of atoms with some parameters dependant on the type of amino acids. 
A more detailed explanation of these energy computations can be found in \cite[Chapter 6-7]{BurkowskiForbesJ2015CaVT}.
\index{dead end elimination}
We finally used a strategy (known as `dead end elimination') to reduce the size of the rotamer sets associated with each amino acid. 
The basic idea behind this strategy is that a rotamer can be removed from its rotamer set if there is another rotamer in that set that gives a better energy value regardless of the rotamer selections for the neighbouring amino acids. 
Among various approaches for the dead end elimination, we followed the Goldstein's criteria \cite{Goldstein:94}.

Let $\cU$ be a side-chain conformation of a protein.
The energy of the conformation $\cU$ is
\[
E(\cU) = \sum_{i=1}^{n_0} E_{\text{self}} (u_i) + \sum_{i=1}^{n_0-1}\sum_{j=i+1}^{n_0} E_{\text{pair}} (u_i,u_j) , 
\]
where $u_i$ is a side-chain conformation of an amino acid,
$E_{\text{self}} (u_i)$ is the energy corresponding to $u_i$ and the backbone, 
and $E_{\text{pair}} (u_i,u_j)$ is the energy formed by $u_i$ and  $u_j$, a rotamer associated with a neighbouring amino acid. 
In our formulation, we placed $E_{\text{self}} (u_i)$ along the diagonal of $E$ and $E_{\text{pair}} (u_i,u_j)$ on the appropriate off-diagonal positions of $E$ as shown in \Cref{sec:Problem_Formulation}.

\subsubsection{Removing Collisions}

We typically observe some very large elements in $E$. This is due to the collisions between rotamers and they are indicated by huge values $E_{i,j} >>0$ that are often greater than $10^{10}$. These huge values occur due to a part of the Lennard-Jones potential formula that involves the Euclidean distance between two distinct rotamers that goes to the denominator of a fraction.

In general, having very large values in data is prone to numerical instabilities. 
If every nonzero elements of $E$ are large, the usual approach is to scale $E$ to avoid large values. However, the matrix $E$ often has elements that are more than $10$ digits as well as elements that are $1$ digit. 
When there is a large discrepancy among the elements of $E$, 
scaling~$E$ would make the relatively small values close to $0$ and lead to loss of precision in the solution. 
However, this ill-posed data does not take place as a problem in our implementation.  
Recall that we update the $Y$ iterate~\eqref{Y_update_complete_sqaure} as follows:
\[
\begin{array}{rcl}
Y^{k+1} 
&=&
 \cP_{\cY} \left( G_{\hat{\cJ}^c } \left( V R^{k+1} V^T - \frac{1}{\beta} (\hat{E}+Z^{k+\frac{1}{2}} ) \right)  \right)
\\&=&
 \cP_{\cY} \left( G_{\hat{\cJ}^c } \left( -\frac{1}{\beta} \hat{E} + \left[ V R^{k+1} V^T - \frac{1}{\beta} Z^{k+\frac{1}{2}} \right] \right)  \right).
\end{array}
\]
For simplicity, we let $T:= -\frac{1}{\beta} \hat{E} + \left[ V R^{k+1} V^T - \frac{1}{\beta} Z^{k+\frac{1}{2}} \right]$.
If  the $(\hat{i},\hat{j})$-th element of $\hat{E} = \BlkDiag(0,E)$ is very large, the projection~$\cP_{\cY} $ sets the $(\hat{i},\hat{j})$-element of~$T$ to $0$ since $T_{\hat{i},\hat{j}}<<0$.
Hence, for those positions $(\hat{i},\hat{j})$ with very large energy values, the constraint~$Y_{\hat{i},\hat{j}}=0$ is implicitly imposed.
We can interpret this as having implicit gangster constraints on these elements.
Consequently, the large elements do not contribute to the objective value since~$\hat{E}_{\hat{i},\hat{j}}Y_{\hat{i},\hat{j}}=0$.

\index{$N_E$}

We can also take advantage of large values in the data to increase the number of the gangster indices (eliminate edges in the graph).
\index{collisions}
\begin{lemma}
\label{lemma:colisionGangster}
Suppose that $x$ is feasible for (\IQPp), and let 
$u = x^TEx$ be its objective value.
Let $N_E = \sum_{\{(i,j) : E_{(i,j)}<0\}}E_{i,j}$ and suppose that 
\[
E_{i_0,j_0} > u-N_E, \text{  for some   }  i_0,j_0
\]
holds. Then for any optimal solution $x^*$ to (\IQPp), we have $x^*_{i_0}x^*_{j_0} = 0$.
\end{lemma}
\begin{proof}
Let $x^*$ be an optimal solution to (\IQPp).
Let $U^*$ be the set of indices formed by the positive entries of 
$\begin{pmatrix} 1 \\ x^* \end{pmatrix} 
\begin{pmatrix} 1 \\ x^* \end{pmatrix} ^T$.
We note that, for any index set $S$, we have
\[
\sum_{(i,j) \in S } E_{i,j}
=
\sum_{(i,j) \in S \cap \{(i,j):E_{i,j} \ge 0\} } E_{i,j}
+\sum_{(i,j) \in S \cap \{(i,j):E_{i,j} < 0\} } E_{i,j}
\ge 0 + N_E  = N_E.
\]
Suppose to the contrary that $x^*$ holds $x^*_{i_0}x^*_{j_0} = 1$, i.e., $x^*_{i_0}=x^*_{j_0} = 1$. 
Then we reach the following contradiction:
\[
p^*_{\IQPp} = \<x^*,Ex^*\>  =  E_{i_0,j_0} + \left(  E_{i_0,j_0}  + \sum\limits_{(i,j)\in U^* \setminus \{ (i_0,j_0) \}} E_{i,j} \right) \ \ge \ E_{i_0j_0}  +N_E > u .
\]
\end{proof}

\begin{corollary}
\label{coro:CollisionGangsterStrip}
Let $i_0$ be an index such that
$E_{i_0,i_0} > u-N_E$, where $u,N_E$ defined in \Cref{lemma:colisionGangster}.
Then, for any optimal solution $x^*$ to (\IQPp), we have
\[
Y_{x^*} :=\begin{pmatrix} 1 \\ x^* \end{pmatrix}
\begin{pmatrix} 1 \\ x^* \end{pmatrix}^T
\in \left\{Y \in \bS^{n_0+1}: Y(:,i_0) = 0, \ Y(i_0,:)=0 \right\}.
\]
\end{corollary}

\begin{proof}
Let $i_0$ be an index such that $E_{i_0,i_0} > u-N_E$.
Then $x_{i_0}^* = 0$ by \Cref{lemma:colisionGangster}.
We note that $Y_{x^*}$ is a positive semidefinite matrix.
If a diagonal entry of a positive semidefinite is zero, then its corresponding column and row must be $0$.
\end{proof}

By \Cref{lemma:colisionGangster} and \Cref{coro:CollisionGangsterStrip}, if we detect entries $i_0,j_0$ with the property  $E_{i_0,j_0} > u-N_E$, then we
may strengthen the model by adding  
the constraints 
\[
\mathcal{K} 
= \left\{ 
Y \in \bS^{n_0+1} : 
\begin{array}{ll}
Y(i_0,j_0) = Y(j_0,i_0) = 0, & \text{ for } i_0 \ne j_0 \text{ such that } \ E_{i_0,j_0} > u-N_E  \\
Y(:,i_0) = 0,\ Y(i_0,:) = 0, &  \text{ for } i_0 \text{ such that } \ E_{i_0,i_0} > u-N_E
\end{array}
\right\} .
\]
This can be easily realized by adding more members to the gangster index set $\hat{\cJ}$.

\subsection{Experiments with Real-World Data}
\label{sec:Experiments}

In this section we provide numerical experiments with real-world data from Protein Data Bank and discuss the strengths of the \DNN relaxation. 
We observe the useful aspects of the \textbf{DNN} relaxation through the numerical experiments. 
The \DNN relaxation provides an effective treatment for 
avoiding numerical instabilities that originate from the large positive values in the data matrix $E$. 
Moreover, we observe that the \DNN relaxation provides superior performance over the \SDP relaxation.

We select instances listed in  \cite{CanutescuAdrianA2003Agaf} with proteins that have up to $300$ amino acids. 
All instances in \Cref{table:proteinRepresentative} are tested using 
MATLAB version 2021a on Dell XPS 8940 with 11th Gen Intel(R) Core(TM) i5-11400 @ 2.60GHz 2.60 GHz with 32 Gigabyte memory.
The following list defines the column headers used in \Cref{table:proteinRepresentative}; we use the same headers to the additional numerical experiments that are displayed in \Cref{sec:AdditionalNumerics}.
\begin{enumerate}[itemsep=-1mm]
\item {\bf name}: instance name;
\item $\bm{p}$: the number of amino acids;
\item $\bm{n_0}$: the total number of rotamers;
\item {\bf lbd}: the lower bound obtained by running \rPRSMp;
\item {\bf ubd}: the upper bound obtained by running \rPRSMp;
\item {\bf rel-gap}: relative gap of each instance using \rPRSMp, 
where 
\begin{equation*}
\label{def:rel_gap}
\text{relative gap} := 2\ \frac{ | \text{best feasible upper
bound}- \text{best lower bound} | }{\ | \text{best feasible upper 
bound}+\text{best lower bound}+1 | } ;
\end{equation*}
\item {\bf iter}: number of iterations used by \rPRSM with tolerance 
	$\epsilon = 10^{-10}$;
	\item {\bf time(sec)}: CPU time (in seconds) used by \rPRSMp.
\end{enumerate}
\begin{table}[h!]
\centering
\small
\begin{tabular}{|cccc||ccc||cc|} \hline
\multicolumn{4}{|c||}{\textbf{Problem Data}} & \multicolumn{3}{|c||}{\textbf{Numerical Results}} & \multicolumn{2}{|c|}{\textbf{Timing}}  \cr\hline
\multicolumn{1}{|c|}{\#} & \multicolumn{1}{|c|}{\textbf{name}} & \multicolumn{1}{|c|}{$\bm{p}$} & \multicolumn{1}{|c||}{$\bm{n_0}$} & \multicolumn{1}{|c|}{\textbf{lbd}} & \multicolumn{1}{|c|}{\textbf{ubd}} & \multicolumn{1}{|c||}{\textbf{rel-gap}} & \multicolumn{1}{|c|}{\textbf{iter}} & \multicolumn{1}{|c|}{\textbf{time(sec)}}  \cr\hline
10  & 2IGD  & 50 & 126 & -78.50608 & -78.50608 & 5.39611e-15 & 500 &    19.43   \cr 
20  & 1VQB  & 75 & 406 & -96.94940 & -96.94940 & 4.34568e-14 & 900 &   179.35   \cr 
30  & 2ACY  & 84 & 580 & -146.32254 & -146.32254 & 1.06468e-14 & 7800 &  2610.24   \cr 
40  & 2TGI  & 100 & 355 & -14.03554 & -14.03554 & 2.46249e-13 & 1300 &   136.30   \cr 
50  & 2SAK  & 111 & 214 & -239.86975 & -239.86975 & 1.08995e-12 & 500 &    25.50   \cr 
60  & 2CPL  & 132 & 819 & -284.97180 & -284.97180 & 9.75693e-15 & 5900 &  3292.98   \cr 
70  & 1CV8  & 146 & 730 & -213.13554 & -213.13554 & 3.28738e-13 & 5600 &  2572.99   \cr 
80  & 2ENG  & 162 & 867 & 82.01797 & 82.01797 & 1.33295e-13 & 14200 &  8274.48   \cr 
90  & 1A7S  & 179 & 524 & -239.78218 & -239.78218 & 1.00542e-14 & 1200 &   314.57   \cr 
100  & 1MRJ  & 208 & 1178 & -295.13711 & -295.13711 & 1.70740e-13 & 2300 &  2421.15   \cr 
110  & 1EZM  & 239 & 1497 & -217.36581 & -217.36581 & 3.49620e-13 & 2300 &  3876.18   \cr 
120  & 1SBP  & 256 & 1704 & -271.08838 & -271.08838 & 3.59996e-14 & 40000 & 609487.29   \cr 
130  & 3PTE  & 284 & 2006 & 161.17216 & 161.17216 & 5.09815e-15 & 13500 & 250604.17   \cr\hline
\end{tabular}
\caption{Computational results on selected PDB instances}
\label{table:proteinRepresentative}
\end{table}

\paragraph{Discussion}

We observe from the last two columns of \Cref{table:proteinRepresentative} that many instances are solved within good relative gaps. 
In fact, most of the instances display relative gaps that are essentially~$0$. 
We recall from \eqref{eq:nearest_sol_LP} that we obtain the upper bounds via finding feasible solutions to (\IQPp).
We recall from \Cref{sec:Upper_Bound} that we obtain the upper bounds via finding feasible solutions to (\IQPp).
That we have the relative gap essentially $0$ grants us the attainment of the \emph{globally optimal} solutions to the \SCP problem. 
Approaches involving heuristic algorithms do not provide a natural means of certifying optimality, relying solely on a comparison of the rotameric solution with naive $\chi_1$ and $\chi_2$ angles from the PDB while ignoring optimality of the discretized solution. 
We highlight that we provide not only the globally optimal solutions but also a way to certify their optimality.

\subsubsection{A Tighter Relaxation}
\label{sec:AtightRelaxation}

We illustrate the strengths of the \DNN relaxation by computing the near optimal values of the \DNN relaxaion and the \SDP relaxation.
In our test, we selected five small instances. As discussed above, some elements of the energy matrix $E$ are typically very large due to the collisions in rotamers, typically at least $10$ digits. 
These cause numerical difficulties when a standard interior point solver is used. 
Hence, in our test, we set the entries $E_{i,j} = \min \{10^4, E_{i,j}\}$, $\forall i,j$, in order to avoid the difficulties from having these large elements. 
We used the \rPRSM for \DNN relaxation and used SDPT3\footnote{\url{https://www.math.cmu.edu/~reha/sdpt3.html}, version
SDPT3 4.0, \cite{ToddTohTut:96b}.} for solving the \SDP relaxation. 
\begin{table}[h]
\centering
\small
\begin{tabular}{|cc|c|c|} \hline
\textbf{problem $\#$} & \textbf{instance} & \DNN \textbf{relaxation} & \SDP \textbf{relaxation}  \cr\hline
1  & 1AIE  & -46.96 & -2460.53  \cr 
2  & 2ERL  & 55.33 & -18241.26  \cr 
3  & 1CBN  & -40.43 & -22380.58  \cr 
4  & 1RB9  & -76.97 & -23936.35  \cr 
5  & 1BX7  & 16.96 & -23965.88  \cr\hline
\end{tabular}
\caption{The solver optimal values of the \DNN and \SDP relaxations on selected instances}
\label{table:proteinDNNSDP}
\end{table}  
The displayed values in \Cref{table:proteinDNNSDP}  are the best lower bounds found from the \rPRSM and the optimal values reported by SDPT3. We observe in  \Cref{table:proteinDNNSDP}
that the \DNN relaxation shows superior performances over the \SDP relaxations in the relaxation values; the \DNN relaxation for the \SDP problem provides a much tighter relaxation than the \SDP relaxation.

\section{Conclusions}

We presented a simple way of formulating the relaxation of 
the \SCP problem.
We began by formulating the \SCP problem into an \IQP and derived the
facially reduced \SDP relaxation.
We then identified some redundant constraints to the \IQP to complete the \DNN relaxation.
\FR allowed for a natural splitting of the variables and provided a perfect environment for using splitting methods. 
Hence we adopted the \rPRSM to solve the \DNN relaxation of the \SCP problem. 
We illustrated the efficiency of our approach using data from the Protein Data Bank. In particular, we solved many instances chosen from the Protein Data Bank to optimality.

\cleardoublepage
\phantomsection
\addcontentsline{toc}{section}{Index}
\printindex
\label{ind:index}

\cleardoublepage
\phantomsection
\addcontentsline{toc}{section}{Bibliography}
\bibliography{bibarXiv.bib}



\appendix

\section{Additional Numerics}
\label{sec:AdditionalNumerics}

\begin{table}[H]
\centering
\scriptsize
\begin{tabular}{|cccc||ccc||cc|} \hline
\multicolumn{4}{|c||}{\textbf{Problem Data}} & \multicolumn{3}{|c||}{\textbf{Numerical Results}} & \multicolumn{2}{|c|}{\textbf{Timing}}  \cr\hline
\multicolumn{1}{|c|}{\#} & \multicolumn{1}{|c|}{\textbf{name}} & \multicolumn{1}{|c|}{$\bm{p}$} & \multicolumn{1}{|c||}{$\bm{n_0}$} & \multicolumn{1}{|c|}{\textbf{lbd}} & \multicolumn{1}{|c|}{\textbf{ubd}} & \multicolumn{1}{|c||}{\textbf{rel-gap}} & \multicolumn{1}{|c|}{\textbf{iter}} & \multicolumn{1}{|c|}{\textbf{time(sec)}}  \cr\hline
1  & 1AIE  & 26 & 34 & -46.95892 & -46.95892 & 1.04802e-15 & 200 &     0.10   \cr 
2  & 2ERL  & 34 & 103 & 55.33285 & 55.33284 & 1.17985e-12 & 200 &     5.85   \cr 
3  & 1CBN  & 37 & 112 & -40.42751 & -40.42751 & 1.68402e-14 & 300 &     7.77   \cr 
4  & 1RB9  & 41 & 105 & -76.96501 & -76.96501 & 7.11964e-13 & 1000 &    26.39   \cr 
5  & 1BX7  & 41 & 99 & 16.96026 & 16.96026 & 5.21525e-12 & 300 &     7.25   \cr 
6  & 2FDN  & 42 & 51 & -59.43091 & -59.43092 & 3.71094e-14 & 200 &     0.04   \cr 
7  & 1MOF  & 46 & 94 & -79.05580 & -79.05580 & 3.52629e-12 & 200 &     4.03   \cr 
8  & 1CTF  & 47 & 74 & -97.18893 & -97.18893 & 4.64633e-13 & 200 &     2.81   \cr 
9  & 1NKD  & 50 & 199 & -51.78466 & -51.78466 & 4.40639e-12 & 2680 &   192.65   \cr 
10  & 2IGD  & 50 & 126 & -78.50608 & -78.50608 & 5.39611e-15 & 500 &    14.67   \cr 
11  & 2SN3  & 53 & 112 & -5.56818 & -5.56818 & 6.73872e-13 & 700 &    16.77   \cr 
12  & 1MSI  & 54 & 112 & -87.46958 & -87.46958 & 1.72043e-13 & 700 &    19.39   \cr 
13  & 1AHO  & 54 & 140 & 24.66925 & 24.66925 & 4.19224e-14 & 1500 &    56.22   \cr 
14  & 1COR  & 60 & 131 & 15.58314 & 15.58314 & 4.58637e-12 & 1000 &    32.31   \cr 
15  & 1CTJ  & 61 & 258 & -103.32705 & -103.32705 & 1.64217e-12 & 1872 &   162.80   \cr 
16  & 1RZL  & 65 & 121 & 17.26470 & 17.26470 & 1.22992e-11 & 2468 &    68.52   \cr 
17  & 1TIF  & 66 & 614 & -155.17859 & -155.17859 & 4.69196e-14 & 1000 &   350.89   \cr 
18  & 1BDO  & 69 & 221 & -136.29933 & -136.29933 & 8.93377e-15 & 1000 &    75.06   \cr 
19  & 1OPD  & 70 & 112 & -139.64632 & -139.64632 & 1.18233e-13 & 300 &     5.98   \cr 
20  & 1VQB  & 75 & 406 & -96.94940 & -96.94940 & 4.34568e-14 & 900 &   147.36   \cr 
21  & 1IUZ  & 75 & 221 & -150.88238 & -150.88238 & 1.25791e-14 & 3200 &   227.45   \cr 
22  & 1ABA  & 76 & 376 & -137.59962 & -137.59963 & 9.05546e-15 & 600 &    88.43   \cr 
23  & 1FNA  & 76 & 131 & -172.01313 & -172.01313 & 3.64100e-14 & 800 &    23.32   \cr 
24  & 1CYO  & 78 & 220 & -75.36668 & -75.36668 & 1.36739e-14 & 700 &    48.50   \cr 
25  & 1FUS  & 79 & 302 & -4.66627 & -4.66627 & 1.11145e-12 & 3000 &   312.35   \cr 
26  & 2MCM  & 80 & 123 & -135.14024 & -135.14024 & 8.30816e-13 & 400 &    10.30   \cr 
27  & 1SVY  & 80 & 147 & -141.92437 & -141.92437 & 6.21219e-13 & 400 &    14.51   \cr 
28  & 1A68  & 81 & 424 & -178.12555 & -178.12555 & 2.54581e-15 & 1500 &   249.80   \cr 
29  & 1YCC  & 84 & 223 & -79.21270 & -79.21270 & 2.11079e-12 & 955 &    66.26   \cr 
30  & 2ACY  & 84 & 580 & -146.32254 & -146.32254 & 1.06468e-14 & 7800 &  2175.04   \cr 
31  & 1BM8  & 85 & 687 & -119.54537 & -119.54537 & 2.02428e-14 & 1300 &   509.88   \cr 
32  & 1BKF  & 89 & 339 & -170.80514 & -170.80514 & 1.60935e-14 & 1000 &   117.73   \cr 
33  & 3CYR  & 91 & 137 & -144.06405 & -144.06405 & 2.48290e-12 & 1900 &    52.09   \cr 
34  & 3VUB  & 92 & 544 & -229.38312 & -229.38312 & 7.41813e-16 & 1400 &   349.67   \cr 
35  & 1JER  & 96 & 462 & -120.78401 & -120.78400 & 1.15131e-12 & 3232 &   633.90   \cr 
36  & 2HBG  & 97 & 275 & -178.42210 & -178.42210 & 2.70839e-13 & 500 &    42.98   \cr 
37  & 1POA  & 97 & 470 & 278.08280 & 278.08280 & 2.02964e-12 & 5463 &  1099.55   \cr 
38  & 1C52  & 99 & 256 & -223.31096 & -223.31096 & 2.41281e-15 & 2700 &   203.46   \cr 
39  & 2A0B  & 99 & 642 & -161.45228 & -161.45228 & 1.75494e-16 & 5200 &  1800.90   \cr 
40  & 2TGI  & 100 & 355 & -14.03554 & -14.03554 & 2.46249e-13 & 1300 &   153.95   \cr\hline
\end{tabular}
\caption{\ Computation results on selected PDB instances up to $100$ amino acids}
\label{table:SCPaminoupto99}
\end{table}

\begin{table}[H]
\centering
\scriptsize
\begin{tabular}{|cccc||ccc||cc|} \hline
\multicolumn{4}{|c||}{\textbf{Problem Data}} & \multicolumn{3}{|c||}{\textbf{Numerical Results}} & \multicolumn{2}{|c|}{\textbf{Timing}}  \cr\hline
\multicolumn{1}{|c|}{\#} & \multicolumn{1}{|c|}{\textbf{name}} & \multicolumn{1}{|c|}{$\bm{p}$} & \multicolumn{1}{|c||}{$\bm{n_0}$} & \multicolumn{1}{|c|}{\textbf{lbd}} & \multicolumn{1}{|c|}{\textbf{ubd}} & \multicolumn{1}{|c||}{\textbf{rel-gap}} & \multicolumn{1}{|c|}{\textbf{iter}} & \multicolumn{1}{|c|}{\textbf{time(sec)}}  \cr\hline
41  & 3NUL  & 101 & 285 & -154.87542 & -154.87542 & 1.28046e-15 & 2300 &   307.34   \cr 
42  & 1WHI  & 101 & 298 & -247.13457 & -247.13457 & 6.94375e-14 & 1500 &   199.52   \cr 
43  & 1PDO  & 104 & 453 & -188.29848 & -188.29848 & 9.10541e-12 & 5754 &  1456.33   \cr 
44  & 3LZT  & 105 & 530 & -48.81821 & -48.81821 & 8.48591e-13 & 1100 &   300.50   \cr 
45  & 1DHN  & 105 & 519 & -133.77464 & -133.77464 & 1.35468e-13 & 2000 &   535.83   \cr 
46  & 1KUH  & 106 & 580 & -155.56590 & -155.56590 & 2.18536e-15 & 2296 &   743.57   \cr 
47  & 1ECA  & 108 & 655 & -169.74717 & -169.74717 & 1.66944e-16 & 25200 & 12563.89   \cr 
48  & 1BFG  & 108 & 410 & -191.73261 & -191.73262 & 8.54577e-14 & 900 &   210.84   \cr 
49  & 1RIE  & 108 & 930 & -117.91809 & -117.91809 & 1.57208e-14 & 20200 & 17809.01   \cr 
50  & 2SAK  & 111 & 214 & -239.86975 & -239.86975 & 1.08995e-12 & 500 &    37.26   \cr 
51  & 1BGF  & 112 & 1180 & -239.65571 & -239.65571 & 1.52549e-13 & 56400 & 71503.54   \cr 
52  & 2END  & 118 & 707 & -8.22833 & -8.22833 & 1.08596e-12 & 16100 &  8511.24   \cr 
53  & 2SNS  & 119 & 634 & 620.86546 & 620.86546 & 1.79304e-14 & 6900 &  3082.12   \cr 
54  & 1BD8  & 121 & 347 & -219.12419 & -219.12419 & 9.42666e-12 & 4970 &   760.81   \cr 
55  & 1NPK  & 122 & 709 & -205.56059 & -205.56059 & 6.77231e-13 & 59075 & 31212.37   \cr 
56  & 1A6M  & 124 & 613 & -55.41007 & -55.41008 & 4.93096e-14 & 22800 &  7608.82   \cr 
57  & 2RN2  & 127 & 830 & -198.37189 & -198.37189 & 1.41057e-13 & 6073 &  4053.13   \cr 
58  & 1RCF  & 130 & 733 & -86.59895 & -86.59775 & 1.38011e-05 & 100000 & 56927.20   \cr 
59  & 1LCL  & 131 & 1246 & -217.16433 & -217.16433 & 2.53317e-14 & 3800 &  4821.11   \cr 
60  & 2CPL  & 132 & 819 & -284.97180 & -284.97180 & 9.75693e-15 & 5900 &  3329.39   \cr 
61  & 1VHH  & 133 & 844 & -21.33604 & -21.33604 & 3.59566e-14 & 3200 &  1843.96   \cr 
62  & 1BJ7  & 135 & 917 & -64.37915 & -64.37915 & 5.69493e-14 & 11300 &  8946.94   \cr 
63  & 119L  & 136 & 970 & -234.21535 & -234.21535 & 8.01617e-14 & 34200 & 30890.87   \cr 
64  & 1RA9  & 136 & 1018 & -185.07235 & -185.07235 & 5.13076e-14 & 4400 &  4839.16   \cr 
65  & 1L58  & 137 & 962 & -285.60167 & -285.60167 & 1.31131e-14 & 15600 & 13812.60   \cr 
66  & 2ILK  & 142 & 708 & -121.02712 & -121.02712 & 1.82770e-13 & 4700 &  2750.13   \cr 
67  & 1KOE  & 144 & 710 & -13.87537 & -13.87537 & 1.27269e-11 & 4124 &  2490.08   \cr 
68  & 1HA1  & 146 & 538 & -213.93793 & -213.93793 & 1.44469e-13 & 3700 &  1229.31   \cr 
69  & 1CEX  & 146 & 415 & 174.95279 & 174.95279 & 2.40438e-11 & 11447 &  2426.49   \cr 
70  & 1CV8  & 146 & 730 & -213.13554 & -213.13554 & 3.28738e-13 & 5600 &  3442.13   \cr 
71  & 153L  & 149 & 846 & -170.13061 & -170.13061 & 3.03488e-13 & 2100 &  1554.46   \cr 
72  & 1BS9  & 150 & 935 & 103.16569 & 103.16569 & 1.31052e-13 & 2500 &  1736.57   \cr 
73  & 2PTH  & 151 & 1198 & -190.97344 & -190.97344 & 1.39085e-13 & 1900 &  2233.17   \cr 
74  & 1XNB  & 151 & 1233 & -147.30040 & -147.30040 & 2.69217e-15 & 13300 & 16562.76   \cr 
75  & 1AQB  & 152 & 713 & 29.24537 & 29.24537 & 9.30418e-14 & 39300 & 17795.39   \cr 
76  & 1LBU  & 152 & 1225 & 38.14603 & 38.14603 & 1.91397e-13 & 9900 & 11673.18   \cr 
77  & 1KID  & 153 & 653 & -351.91160 & -351.91160 & 2.90337e-15 & 6600 &  2607.24   \cr 
78  & 1CHD  & 154 & 489 & -164.21510 & -164.21510 & 3.27846e-14 & 19300 &  4097.50   \cr 
79  & 1AMM  & 158 & 1480 & -288.62671 & -288.62671 & 2.75245e-15 & 3300 &  5793.13   \cr 
80  & 2ENG  & 162 & 867 & 82.01797 & 82.01797 & 1.33295e-13 & 14200 &  8284.65   \cr 
81  & 1G3P  & 165 & 921 & -70.30769 & -70.30769 & 6.66312e-14 & 7000 &  4469.99   \cr 
82  & 1THV  & 167 & 902 & 5.12749 & 5.12749 & 4.63732e-12 & 4200 &  2637.88   \cr 
83  & 1PPN  & 170 & 1259 & -56.69346 & -56.69346 & 1.23365e-13 & 11589 & 14139.22   \cr 
84  & 1IAB  & 173 & 775 & 321.20652 & 321.20652 & 2.04964e-14 & 26500 & 13017.74   \cr 
85  & 1DIN  & 175 & 1110 & -264.73564 & -264.73548 & 5.84356e-07 & 100000 & 93357.26   \cr 
86  & 2AYH  & 176 & 1269 & 8428.18154 & 6089367.83709 & 1.99447e+00 & 100000 & 135879.29   \cr 
87  & 1ZIN  & 177 & 853 & -353.00431 & -353.00431 & 3.18384e-14 & 23800 & 13742.52   \cr 
88  & 1BYI  & 177 & 818 & -242.78881 & -242.78881 & 2.33646e-14 & 2400 &  1298.65   \cr 
89  & 2BAA  & 178 & 1165 & -43.77265 & -43.77265 & 1.95480e-12 & 4600 &  4785.88   \cr 
90  & 1A7S  & 179 & 524 & -239.78218 & -239.78218 & 1.00542e-14 & 1200 &   284.88   \cr 
91  & 1WAB  & 183 & 1063 & -317.46713 & -317.46713 & 9.40337e-14 & 8500 &  7357.75   \cr 
92  & 1MUN  & 185 & 1047 & -378.01261 & -378.01261 & 1.15635e-14 & 9500 &  7883.00   \cr 
93  & 1LST  & 192 & 946 & -244.76861 & -244.76861 & 1.28627e-14 & 32300 & 21374.44   \cr 
94  & 1GCI  & 194 & 1052 & -205.63185 & -205.63185 & 2.79899e-14 & 10300 &  8885.03   \cr 
95  & 3CLA  & 198 & 857 & -26.72768 & -26.72768 & 9.89051e-14 & 3900 &  2287.99   \cr\hline
\end{tabular}
\caption{\ Computation results on selected PDB instances up to $200$ amino acids}
\end{table}

\begin{table}[H]
\centering
\scriptsize
\begin{tabular}{|cccc||ccc||cc|} \hline
\multicolumn{4}{|c||}{\textbf{Problem Data}} & \multicolumn{3}{|c||}{\textbf{Numerical Results}} & \multicolumn{2}{|c|}{\textbf{Timing}}  \cr\hline
\multicolumn{1}{|c|}{\#} & \multicolumn{1}{|c|}{\textbf{name}} & \multicolumn{1}{|c|}{$\bm{p}$} & \multicolumn{1}{|c||}{$\bm{n_0}$} & \multicolumn{1}{|c|}{\textbf{lbd}} & \multicolumn{1}{|c|}{\textbf{ubd}} & \multicolumn{1}{|c||}{\textbf{rel-gap}} & \multicolumn{1}{|c|}{\textbf{iter}} & \multicolumn{1}{|c|}{\textbf{time(sec)}}  \cr\hline
96  & 1AL3  & 201 & 1077 & 119.66598 & 119.66598 & 3.39407e-14 & 12500 & 10188.87   \cr 
97  & 1ARB  & 202 & 1466 & -61.52823 & -61.52823 & 3.41363e-14 & 8900 & 14632.82   \cr 
98  & 1XJO  & 202 & 776 & -171.92443 & -171.92443 & 8.24179e-15 & 3700 &  1455.50   \cr 
99  & 1NLS  & 203 & 1060 & -297.73578 & -297.73578 & 5.33677e-15 & 2500 &  1976.08   \cr 
100  & 1MRJ  & 208 & 1178 & -295.13711 & -295.13711 & 1.70740e-13 & 2300 &  2149.63   \cr 
101  & 1OAA  & 208 & 854 & -317.83422 & -317.83422 & 1.44174e-12 & 3842 &  1823.52   \cr 
102  & 2DRI  & 210 & 906 & -398.45564 & -398.45564 & 2.56465e-15 & 6200 &  3225.99   \cr 
103  & 2CBA  & 223 & 1018 & -86.52145 & -86.52145 & 5.34000e-14 & 3400 &  2407.24   \cr 
104  & 2POR  & 224 & 1304 & -83.22221 & -83.22221 & 5.55044e-14 & 6700 &  8044.39   \cr 
105  & 3SEB  & 224 & 1412 & 77.15838 & 77.15852 & 1.84867e-06 & 100000 & 137194.81   \cr 
106  & 1MLA  & 227 & 1322 & -484.10542 & -484.10542 & 1.68910e-14 & 62900 & 75257.79   \cr 
107  & 1DCS  & 232 & 1170 & -342.68600 & -342.68600 & 1.39133e-14 & 8000 &  7459.07   \cr 
108  & 1AKO  & 234 & 1387 & -244.65691 & -244.65691 & 1.18251e-14 & 7400 &  9809.00   \cr 
109  & 1PDA  & 239 & 891 & -423.50226 & -423.50226 & 4.96037e-15 & 9100 &  4520.68   \cr 
110  & 1EZM  & 239 & 1497 & -217.36581 & -217.36581 & 3.49620e-13 & 2300 &  3919.92   \cr 
111  & 1C3D  & 243 & 1679 & -400.69876 & -400.69876 & 1.04846e-14 & 22100 & 134094.53   \cr 
112  & 1RHS  & 244 & 1973 & -341.20443 & -341.20443 & 1.41400e-14 & 7300 & 62136.57   \cr 
113  & 8ABP  & 245 & 1743 & -273.90715 & -273.90716 & 2.27865e-15 & 9000 & 59868.98   \cr 
114  & 1CVL  & 246 & 910 & -537.04249 & -537.04249 & 2.11494e-16 & 14800 &  7522.51   \cr 
115  & 1RYC  & 248 & 1831 & -202.60568 & -202.60568 & 4.81378e-14 & 15200 & 84674.22   \cr 
116  & 1MRP  & 248 & 1648 & -350.97062 & -350.97062 & 1.39088e-14 & 11000 & 34303.23   \cr 
117  & 1IXH  & 252 & 1134 & -289.75241 & -289.75241 & 4.11267e-14 & 1300 &  1087.30   \cr 
118  & 1FNC  & 253 & 1940 & -310.60999 & -310.60999 & 6.54656e-13 & 34321 & 292924.91   \cr 
119  & 1TCA  & 255 & 1062 & -422.15387 & -422.15387 & 4.24994e-14 & 8700 &  6424.87   \cr 
120  & 1SBP  & 256 & 1704 & -271.08838 & -271.08838 & 3.59996e-14 & 40000 & 156330.60   \cr 
121  & 2CTC  & 264 & 1536 & -213.88596 & -213.88596 & 2.17419e-14 & 15100 & 43642.85   \cr 
122  & 1PGS  & 265 & 2190 & -16.14049 & -16.14049 & 2.28785e-12 & 21300 & 269611.15   \cr 
123  & 1MSK  & 271 & 1798 & -162.51007 & -162.50978 & 1.77573e-06 & 100000 & 771330.61   \cr 
124  & 1BG6  & 271 & 784 & -452.62383 & -452.62383 & 3.13620e-15 & 12700 &  4935.11   \cr 
125  & 1ARU  & 271 & 939 & -314.40612 & -314.40589 & 7.15908e-07 & 100000 & 53858.54   \cr 
126  & 1A8E  & 274 & 1096 & -249.85499 & -249.85499 & 3.58741e-14 & 96500 & 78746.74   \cr 
127  & 1AXN  & 278 & 2343 & -300.34291 & -300.34291 & 7.55789e-15 & 12500 & 207625.02   \cr 
128  & 1TAG  & 279 & 1330 & -253.22167 & -253.22167 & 1.68029e-14 & 4300 &  5038.43   \cr 
129  & 1ADS  & 280 & 1560 & 733.91439 & 733.91440 & 1.39319e-13 & 18273 & 65301.22   \cr 
130  & 3PTE  & 284 & 2006 & 161.17216 & 161.17216 & 5.09815e-15 & 13500 & 59169.60   \cr 
131  & 1CEM  & 292 & 2400 & -24.20196 & -24.20196 & 3.85446e-14 & 7000 & 47701.70   \cr\hline
\end{tabular}
\caption{\ Computation results on selected PDB instances up to $300$ amino acids}
\label{table:SCPaminoupto299}
\end{table}

\end{document}